\documentclass[11pt,leqno]{amsart}
\usepackage{amsmath,amssymb,amsthm}
\usepackage{hyperref}

\usepackage{bbm}

\renewcommand{\leq}{\leqslant}

\newcommand{\1}[1]{\operatorname{\textbf{1}}}

\newtheorem{theorem}{Theorem}[section]
\newtheorem{lemma}[theorem]{Lemma}

\newtheorem{proposition}[theorem]{Proposition}

\theoremstyle{definition}

\theoremstyle{remark}

\numberwithin{equation}{section}

\def\fnote#1{\footnote}

\def\ignora#1{}
\def\n3#1{\left\vert  \! \left\vert \! \left\vert \, #1 \, \right\vert \!
  \right\vert \! \right\vert }

\newcommand{\iten}{\ensuremath{\widehat{\otimes}_\varepsilon}}
\newcommand{\pten}{\ensuremath{\widehat{\otimes}_\pi}}

\begin{document}

\title{ Superreflexive tensor product spaces }

\author{ Abraham Rueda Zoca }\address{Universidad de Granada, Facultad de Ciencias. Departamento de An\'{a}lisis Matem\'{a}tico, 18071-Granada
(Spain)} \email{ abrahamrueda@ugr.es}
\urladdr{\url{https://arzenglish.wordpress.com}}



\maketitle

\markboth{ABRAHAM RUEDA ZOCA}{SUPERREFLEXIVE TENSOR PRODUCT SPACES}

\begin{abstract}
The aim of this note is to prove that, given two superreflexive Banach spaces $X$ and $Y$, then $X\pten Y$ is superreflexive if and only if either $X$ or $Y$ is finite-dimensional. In a similar way, we prove that $X\iten Y$ is superreflexive if and only if either $X$ or $Y$ is finite-dimensional.
\end{abstract}

\section{Introduction}

The study of topological and geometrical properties of Banach spaces has attracted the attention of many researches in Functional Analysis, the papers \cite{amrr22,bu-die,dies,rr24,bu-alt} and the references therein are a good sample of this.

One classical result in tensor product spaces is the well known characterisation of the reflexivity in tensor product spaces \cite[Theorem 4.21]{ryan}. Namely, given two reflexive Banach spaces $X$ and $Y$, one of which has the approximation property. The following are equivalent:
\begin{enumerate}
    \item $X\pten Y$ is reflexive.
    \item $(X\pten Y)^*=X^*\iten Y^*$.
    \item Every operator from $X$ to $Y^*$ is compact.
    \item $X^*\iten Y^*$ is reflexive.
\end{enumerate}

The above characterisation yields examples of non-reflexive tensor product spaces even if both factors are reflexive, e.g. $\ell_2\pten \ell_2$, which even contains $\ell_1$ isometrically \cite[Example 2.10]{ryan}. On the other hand, examples of infinite-dimensional Banach spaces $X$ and $Y$ such that $X\pten Y$ is reflexive can be given. For instance, given $1<p,q<\infty$ with $\frac{1}{p}+\frac{1}{q}<1$, then $\ell_p\pten \ell_q$ is reflexive in virtue of Pitt theorem \cite[Theorem 4.23]{ryan}.

In this note we wonder when a projective tensor product can be superreflexive. As being a condition implying the reflexivity, the above mentioned result \cite[Theorem 4.21]{ryan} will imply that superreflexivity must be an uncommon phenomenon. Going further, a natural question at this point is whether there exists a ``non-trivial'' superreflexive tensor product space, that is, if there are infinite-dimensional Banach spaces $X$ and $Y$ such that $X\pten Y$ is superreflexive.

The main aim of this paper is proving that the answer is no. Indeed, the main theorem of the paper is the following.

\begin{theorem}\label{theo:tensosuperreflexive}
Let $X$ and $Y$ be two superreflexive Banach spaces. The following are equivalent:
\begin{enumerate}
    \item $X\pten Y$ is superreflexive.
    \item Either $X$ or $Y$ is finite dimensional.
\end{enumerate}
\end{theorem}

The above result says that the unique possibility for a projective tensor product to be superreflexive is that one of the factors is finite-dimensional, which establishes a big different with the case of reflexivity.

Our methods, which will be focused on studying ultraproducts of injective tensor product, will allow us to derive a similar version for the case of the injective tensor product.

\begin{theorem}\label{theo:tensoinjesuperreflexive}
Let $X$ and $Y$ be two superreflexive Banach spaces. The following are equivalent:
\begin{enumerate}
    \item $X\iten Y$ is superreflexive.
    \item Either $X$ or $Y$ is finite dimensional.
\end{enumerate}
\end{theorem}

After a notation section, which is necessary in order to introduce all the notation about tensor product spaces and ultraproducts, Section~\ref{section:proof} will be devoted to prove the above mentioned theorems.

\section{Notation}

For simplicity we will consider real Banach spaces. We denote by $B_X$ and $S_X$ the closed unit ball and the unit sphere, respectively, of the Banach space $X$. We denote by $L(X, Y)$ the space of all bounded linear operators from $X$ into $Y$. If $Y = \mathbb R$, then $ L(X, \mathbb R)$ is denoted by $X^*$, the topological dual space of $X$. A bounded and symmetric subset $A\subseteq B_{X^*}$ is said to be \textit{1-norming} if $\Vert x\Vert=\sup_{f\in A} f(x)$.

\subsection{Ultrapowers}

Given a sequence of Banach spaces $\{X_n:n\in\mathbb N\}$ we denote 
$$\ell_\infty(\mathbb N,X_n):=\left\{f\colon \mathbb N\longrightarrow \prod\limits_{n\in \mathbb N} X_n: f(n)\in X_n\ \forall n\text{ and }\sup_{n\in \mathbb N}\Vert f(n)\Vert<\infty\right\}.$$
Given a free ultrafilter $\mathcal U$ over $\mathbb N$, consider $c_{0,\mathcal U}(\mathbb N,X_n):=\{f\in \ell_\infty(\mathbb N,X_n): \lim_\mathcal U \Vert f(n)\Vert=0\}$. The \textit{ultrapower of $\{X_n:n\in\mathbb N\}$ with respect to $\mathcal U$} is
the Banach space
$$(X_n)_\mathcal U:=\ell_\infty(\mathbb N,X_n)/c_{0,\mathcal U}(\mathbb N,X_n).$$
We will naturally identify a bounded function $f\colon\mathbb N\longrightarrow \prod\limits_{n\in \mathbb N} X_n$ with the element $(f(n))_{n\in\mathbb N}$. In this way, we denote by $(x_n)_\mathcal U$ or simply by $(x_n)$, if no confusion is possible, the coset in $(X_n)_\mathcal U$ given by $(x_n)_{n\in \mathbb N}+c_{0,\mathcal U}(\mathbb N,(X_n))$.

From the definition of the quotient norm, it is not difficult to prove that $\Vert (x_n)\Vert=\lim_\mathcal U \Vert x_n\Vert$ holds for every $(x_n)\in (X_n)_\mathcal U$. If all the spaces $X_n$ are equal to $X$ we will simply write $X_\mathcal U$.

Observe that, in general, $(X_\mathcal U)^*=(X^*)_\mathcal U$ if and only if $X$ is superreflexive \cite[Theorem 6.4]{hein} (see the next subsection for formal definition). However, it follows that $(X^*)_\mathcal U$ is isometrically a subspace of $(X_\mathcal U)^*$ by the action
$$(x_n^*)(x_n):=\lim_\mathcal U x_n^*(x_n)\ \ (x_n)\in X_\mathcal U, (x_n^*)\in (X^*)_\mathcal U.$$
It follows that $S_{(X^*)_\mathcal U}$ is 1-norming for $X_\mathcal U$. Indeed, we have the following result, which is more general, whose proof is included for convenience and for the sake of completeness.

\begin{proposition}\label{prop:normingultrafilter}
Let $X$ be a Banach space and $\mathcal U$ be a free ultrafilter over $\mathbb N$. Let $A\subseteq B_{X^*}$ which is 1-norming for $X$. Then the set
$$A_\mathcal U:=\{(f_n): f_n\in A\ \forall n\in\mathbb N\}\subseteq (X^*)_\mathcal U$$
is a 1-norming set for $X_\mathcal U$.
\end{proposition}

\begin{proof}
Let $(x_n)\in X_\mathcal U$. Select, for every $n\in\mathbb N$, an element $f_n\in A$ such that $f_n(x_n)>\Vert x_n\Vert-\frac{1}{n}$. Now $(f_n)\in A_\mathcal U$, and it is clear that $(f_n)(x_n)=\lim_\mathcal U f_n(x_n)=\lim_\mathcal U \Vert x_n\Vert=\Vert (x_n)\Vert$, as desired.
\end{proof}

\subsection{Superreflexive Banach spaces}

Given two Banach spaces $X$ and $Y$, we say that $Y$ is \textit{finitely representable} in $X$ if, for every finite dimensional subspace $E$ of $Y$ and every $\varepsilon>0$, there exists a finite dimensional subspace $F$ of $X$ and an onto linear mapping $T:E\longrightarrow F$ such that $\Vert T\Vert \Vert T^{-1}\Vert\leq 1+\varepsilon$. 

Recall that $X$ is said to be \textit{superreflexive} if every Banach space $Y$ which is finitely representable in $X$ must be reflexive. We refer the reader to \cite[Chapter 9]{checos} for background. 

It is known that a Banach space $X$ is superreflexive if and only if $X_\mathcal U$ is reflexive for every free ultrafilter $\mathcal U$ over $\mathbb N$ (see the comment after Theorem 1.3.2 in \cite{grelier}).

Observe also that a Banach space $X$ is superreflexive if and only if $X$ admits an equivalent renorming which is simultaneously uniformly convex and uniformly smooth \cite[Theorem 9.18]{checos}. Even though we will not enter in the formal definition of uniformly convex and uniformly smooth Banach spaces, observe that a Banach space $X$ is uniformly convex (respectively uniformly smooth) if and only if $X^*$ is uniformly smooth (repectively uniformly convex) \cite[Theorem 9.10]{checos}.

This result allows us to obtain the following consequence from the above mentioned \cite[Theorem 9.18]{checos}: a Banach space $X$ is superreflexive if and only if $X^*$ is superreflexive.

\subsection{Tensor product spaces}

The projective tensor product of $X$ and $Y$, denoted by $X \pten Y$, is the completion of the algebraic tensor product $X \otimes Y$ endowed with the norm
$$
\|z\|_{\pi} := \inf \left\{ \sum_{n=1}^k \|x_n\| \|y_n\|: z = \sum_{n=1}^k x_n \otimes y_n \right\},$$
where the infimum is taken over all such representations of $z$. The reason for taking completion is that $X\otimes Y$ endowed with the projective norm is complete if and only if either $X$ or $Y$ is finite dimensional (see \cite[P.43, Exercises 2.4 and 2.5]{ryan}).

It is well known that $\|x \otimes y\|_{\pi} = \|x\| \|y\|$ for every $x \in X$, $y \in Y$, and that the closed unit ball of $X \pten Y$ is the closed convex hull of the set $B_X \otimes B_Y = \{ x \otimes y: x \in B_X, y \in B_Y \}$. 

Observe that the action of an operator $G\colon X \longrightarrow Y^*$ as a linear functional on $X \pten Y$ is given by
$$
G \left( \sum_{n=1}^{k} x_n \otimes y_n \right) = \sum_{n=1}^{k} G(x_n)(y_n),$$
for every $\sum_{n=1}^{k} x_n \otimes y_n \in X \otimes Y$. This action establishes a linear isometry from $ L(X,Y^*)$ onto $(X\pten Y)^*$ (see e.g. \cite[Theorem 2.9]{ryan}). All along this paper we will use the isometric identification $(X\pten Y)^*= L(X,Y^*)$ without any explicit mention.

Recall that given two Banach spaces $X$ and $Y$, the
\textit{injective tensor product} of $X$ and $Y$, denoted by
$X \iten Y$, is the completion of $X\otimes Y$ under the norm given by
\begin{equation*}
   \Vert u\Vert_{\varepsilon}:=\sup
   \left\{
      \sum_{i=1}^n \vert x^*(x_i)y^*(y_i)\vert
      : x^*\in S_{X^*}, y^*\in S_{Y^*}
   \right\},
\end{equation*}
where $u=\sum_{i=1}^n x_i\otimes y_i$ (see \cite[Chapter 3]{ryan} for background).
Observe that, from the very definition, the set $S_{X^*}\otimes S_{Y^*}:=\{x^*\otimes y^*: x^*\in S_{X^*}, y^*\in S_{Y^*}\}\subseteq B_{(X\iten Y)^*}$ is a 1-norming subset for $X\iten Y$.

\section{Proof of the results}\label{section:proof}

Observe that, given two Banach spaces $X$ and $Y$, then $X$ and $Y$ can be seen as subspaces of both $X\pten Y$ and $X\iten Y$. Consequently, in order to $X\pten Y$ or $X\iten Y$ be superreflexive then both $X$ and $Y$ must be superreflexive.

Before exhibiting the proof of Theorems~\ref{theo:tensosuperreflexive} and \ref{theo:tensoinjesuperreflexive} we will need the following lemma.

\begin{lemma}\label{lemma:enjaedual}
Let $X$ and $Y$ be two superreflexive Banach spaces and let $\mathcal U$ be a free ultrafilter over $\mathbb N$. Then the mapping $\phi:X_\mathcal U\iten Y_\mathcal U\longrightarrow (X\iten Y)_\mathcal U$ defined by
$$\phi\left(\sum_{i=1}^p (x_n^i)\otimes (y_n^i)\right):=\left(\sum_{i=1}^p x_n^i\otimes y_n^i \right)$$
defines a linear isometry.
\end{lemma}

\begin{proof} The linearity is immediate, so let us prove that it is an isometry. Let $z\in X_\mathcal U\iten Y_\mathcal U$ with $z\neq 0$. We can assume up to a density argument that $z=\sum_{i=1}^p (x_n^i)\otimes (y_n^i)$ for certain $(x_n^i)\in X_\mathcal U$ and $(y_n^i)\in Y_\mathcal U$. Let $\varepsilon>0$. By the definition of the injective norm we can find $f\in S_{(X_\mathcal U)^*}$ and $g\in S_{(Y_\mathcal U)^*}$ such that
$$\Vert z\Vert -\varepsilon<(f\otimes g)(z).$$
Since $X$ and $Y$ are superreflexive then $(X_\mathcal U)^*=(X^*)_\mathcal U$ and, similarly, $(Y_\mathcal U)^*=(Y^*)_\mathcal U$. Consequently, $f=(f_n)\in S_{(X^*)_\mathcal U}$ and $g:=(g_n)\in S_{(Y^*)_\mathcal U}$. Now we have
\[
\begin{split}
\Vert z\Vert-\varepsilon<(f_n)\otimes (g_n)(z)& =\sum_{i=1}^p (f_n)(x_n^i) (g_n)(y_n^i)\\
& =\sum_{i=1}^p \lim_{n,\mathcal U} f_n(x_n^i) \lim_{n,\mathcal U} g_n(y_n^i)  \\
& = \sum_{i=1}^p \lim_{n,\mathcal U} f_n(x_n^i)g_n(y_n^i)\\
& = \lim_{n,\mathcal U}\sum_{i=1}^n f_n(x_n^i)g_n(y_n^i)
\end{split}
\]
Now observe that we can consider $(f_n\otimes g_n)\in S_{(X\iten Y)^*_\mathcal U}$. Evaluating the above element at $\phi(z)$ we get
\[
\begin{split}
(f_n\otimes g_n)(\phi(z))& =(f_n\otimes g_n)\left(\sum_{i=1}^p x_n^i\otimes y_n^i \right)\\
& =\lim_{n,\mathcal U} (f_n\otimes g_n)\left(\sum_{i=1}^p x_n^i\otimes y_n^i \right) \\
& =\lim_{n,\mathcal U}\sum_{i=1}^p f_n(x_n^i)g_n(y_n^i)
\end{split}
\]
The above proves that
\[\begin{split}\Vert z\Vert-\varepsilon<(f_n)\otimes (g_n)(z)& =(f_n\otimes g_n)(\phi(z))\\
& \leq \Vert (f_n\otimes g_n)\Vert_{((X\iten Y)_\mathcal U)^*} \Vert \phi(z)\Vert_{(X\iten Y)_\mathcal U}.
\end{split}\]
Since $\Vert (f_n\otimes g_n)\Vert_{((X\iten Y)_\mathcal U)^*}=\lim_{n,\mathcal U} \Vert f_n\otimes g_n\Vert_{(X\iten Y)^*}=1$ we infer $\Vert z\Vert-\varepsilon\leq \Vert \phi(z)\Vert$. The arbitrariness of $\varepsilon>0$ implies $\Vert z\Vert\leq \Vert\phi(z)\Vert$. 

In order to prove that $\Vert \phi(z)\Vert\leq \Vert z\Vert$ let $\varepsilon>0$. Since $S_{X^*}\otimes S_{Y^*}$ is 1-norming for $X\iten Y$ we infer that $(S_{X^*}\otimes S_{Y^*})_\mathcal U$ is 1-norming for $(X\iten Y)_\mathcal U$ in virtue of Proposition~\ref{prop:normingultrafilter}. Consequently, we can find two sequences $(h_n)\subseteq S_{X^*}$ and $(j_n)\subseteq S_{Y^*}$ such that

$$\Vert \phi(z)\Vert-\varepsilon<(h_n\otimes j_n)(\phi(z)).$$
Recalling the definition of $\phi(z)$ we get
\[\begin{split}(h_n\otimes j_n)(\phi(z))& =(h_n\otimes j_n)\left(\sum_{i=1}^p x_n^i\otimes y_n^i \right)\\
& =\lim_\mathcal U (h_n\otimes j_n)\left(\sum_{i=1}^p x_n^i\otimes y_n^i\right)\\
& = \lim_\mathcal U \sum_{i=1}^p h_n(x_n^i)j_n(y_n^i)\\
& = \sum_{i=1}^p  \lim_\mathcal U  h_n(x_n^i)j_n(y_n^i)\end{split}\]
in virtue of the linearity of the limit through $\mathcal U$.

On the other hand, if we see $(h_n)\otimes (j_n)\in (X_\mathcal U\iten Y_\mathcal U)^*$ we get
\[
\begin{split}
(h_n)\otimes (j_n)(z)& =(h_n)\otimes (j_n)\left(\sum_{i=1}^p (x_n^i)\otimes (y_n^i)\right)\\
& =\sum_{i=1}^p (h_n)(x_n^i) (j_n)(y_n^i)\\
& =\sum_{i=1}^p \lim_\mathcal U h_n(x_n^i)j_n(y_n^i)
\end{split}
\]
With all the above we get
\[
\begin{split}
\Vert \phi(z)\Vert-\varepsilon<(h_n\otimes j_n)(\phi(z))=(h_n)\otimes (j_n)(z)\leq \Vert (h_n)\otimes (j_n)\Vert_{(X_\mathcal U\iten Y_\mathcal U)^*}\Vert z\Vert.
\end{split}
\]
Now observe that since $h_n\in S_{X^*}$ it follows that $\Vert (h_n)\Vert_{X^*_\mathcal U}=\lim_\mathcal U\Vert h_n\Vert=1$. Analogously we get $\Vert (j_n)\Vert_{Y_\mathcal U^*}=1$. Hence $\Vert (h_n)\otimes (j_n)\Vert_{(X_\mathcal U\iten Y_\mathcal U)^*}=1$ and we get
$$\Vert \phi(z)\Vert-\varepsilon\leq \Vert z\Vert.$$
The arbitrariness of $\varepsilon>0$ implies $\Vert \phi(z)\Vert=\Vert z\Vert$ and the lemma is finished.\end{proof}

Now we can provide the proof of Theorem~\ref{theo:tensosuperreflexive}.

\begin{proof}[Proof of Theorem~\ref{theo:tensosuperreflexive}]
(2)$\Rightarrow$(1). If the dimension of $X$ is $N$, then $X$ is isomorphic to $\ell_1^N$. Consequently, $X\pten Y$ is isomorphic to $\ell_1^N\pten Y=\ell_1^N(Y)$ \cite[Example 2.6]{ryan}, which is superreflexive since $Y$ is superreflexive.

(1)$\Rightarrow$(2). Assume that both $X$ and $Y$ are infinite dimensional. Take any free ultrafilter $\mathcal U$ over $\mathbb N$, and let us prove that $(X\pten Y)_\mathcal U$ is not reflexive. This is equivalent to proving that its dual $Z:=(X\pten Y)_\mathcal U^*$ is not reflexive. Observe that $(L(X,Y^*))_\mathcal U=((X\pten Y)^*)_\mathcal U$ is isometrically a subspace of $Z$. Consequently, $(X^*\iten Y^*)_\mathcal U$ is an isometric subspace of $Z$. By Lemma~\ref{lemma:enjaedual} we infer that $Z$ contains an isometric copy of $(X^*)_\mathcal U\iten (Y^*)_\mathcal U$. 

Let us prove that $Z$ contains an isometric copy of $\ell_2\iten \ell_2$. Indeed, since $X^*$ is infinite dimensional then $\ell_2$ is finitely representable in $X^*$ by Dvoretzky theorem (c.f. e.g. \cite[Theorem 12.3.6]{alka}). By \cite[Proposition 11.1.12]{alka} we get that $\ell_2$ is an isometric subspace of $X^*_\mathcal U$. Similarly $\ell_2$ is an isometric subspace of $Y^*_\mathcal U$. Consequently, $\ell_2\iten \ell_2$ is isometrically a subspace of $X^*_\mathcal U\iten Y^*_\mathcal U$ since the injective tensor product respects subspaces. Consequently, $\ell_2\iten \ell_2$ is isometrically a subspace of $Z$. 

This implies that $Z$ is not reflexive since $\ell_2\iten \ell_2$ is not reflexive (c.f. e.g. \cite[Theorem 4.21]{ryan}).     
\end{proof}

A similar proof to the above one yields also the proof of Theorem~\ref{theo:tensoinjesuperreflexive}.

\begin{proof}[Proof of Theorem~\ref{theo:tensoinjesuperreflexive}]
(1)$\Rightarrow$(2) follows the same ideas than the corresponding implication in Theorem~\ref{theo:tensosuperreflexive}.

To prove that (2)$\Rightarrow$(1) assume that $X$ is finite-dimensional. Then  $X$ is isomorphic to $\ell_\infty^N$. This implies that $X\iten Y$ is isomorphic to $\ell_\infty^N\iten Y=\ell_\infty^N(Y)$ \cite[Section 3.2]{ryan}, from where the superreflexivity of $X\iten Y$ follows.
\end{proof}

\section*{Acknowledgements}

This work was supported by MCIN/AEI/10.13039/501100011033: grant PID2021-122126NB-C31, Junta de Andaluc\'ia: grant FQM-0185 and by Fundaci\'on S\'eneca: ACyT Regi\'on de Murcia: grant 21955/PI/22.


\begin{thebibliography}{999999}

\bibitem {alka} F. Albiac, N.J. Kalton, \textit{Topics in Banach Space Theory}, Springer Inc. (2006).

\bibitem {amrr22} A. Avil\'es, G. Mart\'inez-Cervantes, J. Rodr\'iguez and A. Rueda Zoca, \textit{Topological properties in tensor products of Banach spaces}, J. Funct. Anal. \textbf{283} (2022), article 109688.

\bibitem{bu-die}
Q.~Bu and J.~Diestel, \emph{Observations about the projective tensor product of
  {B}anach spaces. {I}. {$l^p\widehat{\otimes}X,\ 1<p<\infty$}}, Quaest. Math.
  \textbf{24} (2001), no.~4, 519--533. 

\bibitem{dies}
J.~Diestel, \emph{A survey of results related to the {D}unford-{P}ettis
  property}, Proceedings of the {C}onference on {I}ntegration, {T}opology, and
  {G}eometry in {L}inear {S}paces ({U}niv. {N}orth {C}arolina, {C}hapel {H}ill,
  {N}.{C}., 1979), Contemp. Math., vol.~2, Amer. Math. Soc., Providence, R.I.,
  1980, pp.~15--60. 

\bibitem {checos} M. Fabian, P. Habala, P. H\'ajek, V. Montesinos, J. Pelant, V. Zizler, \textit{Functional Analysis and Infinite-Dimensional Geometry}, CMS Books in Mathematics. Springer-Velag New York 2001.

\bibitem {grelier} G. Grelier, \textit{Super weak compactness and its applications to Banach space theory}, PhD thesis, Universidad de Murcia, 2022.
\newblock Available at \emph{DigitUM} with reference
  \href{https://digitum.um.es/digitum/handle/10201/126403}{https://digitum.um.es/digitum/handle/10201/126403}.

\bibitem {hein} S. Heinrich, \textit{Ultraproducts in Banach space theory}, J. Reine Angew. Math. \textbf{313} (1980), 72--104.

\bibitem {rr24} J. Rodr\'iguez and A. Rueda Zoca, \textit{Weak precompactness in projective tensor products}, Idag. Math. \textbf{35} (2024), 60--75.

\bibitem {ryan} R.~A.~Ryan, \textit{Introduction to tensor products of Banach spaces}, Springer Monographs in Mathematics, Springer-Verlag, London, 2002.


\bibitem{bu-alt}
X.~Xue, Y.~Li, and Q.~Bu, \emph{Embedding {$l_1$} into the projective tensor
  product of {B}anach spaces}, Taiwanese J. Math. \textbf{11} (2007), no.~4,
  1119--1125.

\end{thebibliography}
\end{document}